\documentclass[12pt]{article}

\usepackage{graphicx}
\usepackage{enumerate}
\usepackage{amssymb}
\usepackage{amsmath}
\usepackage{mathrsfs}
\usepackage{latexsym}
\usepackage{psfrag}
\usepackage{mathrsfs}
\usepackage{verbatim}
\usepackage[table]{xcolor}
\usepackage{comment}
\usepackage{diagbox}
\usepackage{tabularray}
\usepackage{enumerate, lineno, setspace, float}

\usepackage[utf8]{inputenc}
\usepackage{amsfonts,  amsmath,  amssymb,  amsthm}
\usepackage{graphicx}
\usepackage{comment}
\usepackage{tikz,  pgffor}
\usepackage{hyperref}
\usepackage[numbers, sort&compress]{natbib}
\usepackage[capitalize]{cleveref}
\hypersetup{
    colorlinks=true,        
    linkcolor=blue,         
    citecolor=blue,         
}

\allowdisplaybreaks
\newtheorem{theorem}{Theorem}
\newtheorem{lemma}[theorem]{Lemma}
\newtheorem{proposition}[theorem]{Proposition}

\newtheorem{definition}{Definition} 

\newtheorem{conjecture}{Conjecture}
\newtheorem{question}[conjecture]{Question}

\newcommand{\Href}[1]{\hyperref[#1]{\Cref{#1}}}
\renewcommand{\href}[1]{\hyperref[#1]{\ref{#1}}}

\usepackage{color}

\title{Constructing the antimagic labelings for double stars union paths on three vertices} 
\author{
Wei-Tian Li
\thanks{Department of Applied Mathematics,  National Chung Hsing University,  Taichung 40227,  Taiwan. 
{\tt Email:weitianli@nchu.edu.tw.} Supported by NSTC 112-2115-M-005 -002 -MY2}
\and
Po-Wen Yang
\thanks{Department of Applied Mathematics,  National Chung Hsing University,  Taichung 40227,  Taiwan. 
{\tt Email:std910467@gmail.com} }
}

\date{\today}

\begin{document}
\baselineskip=1.15\baselineskip

\maketitle

\begin{abstract}
For a graph on $m$ edges, a bijective function between the edge set of the graph and $\{1,2,\ldots,m\}$ 
is an antimagic labeling provided that when adding the labels of the edges incident to the same vertex, the sums are pairwise distinct. Hartsfield and Ringel conjectured that every connected graph has antimagic labeling. On the other hand, it is known that for any graph $G$, the disjoint union of $G$ and many $P_3$, a path on 3 vertices, is not antimagic. In this paper, we determined the exact number of $P_3$'s such that the disjoint union of a double star with the number of $P_3$'s is antimagic. 
In addition, we provide some examples of $(1,1)$-antimagic labelings. That is, the antimagic labelings have vertex sums 1 through the number of vertices of the graphs. 
\end{abstract}

{\bf keywords: antimagic, $(a,d)$-antimagic,   disconnected graphs, Pell equation}

\section{Introduction}

All graphs in this paper are simple graphs. 
Let $G$ be a graph on $m$ edges and $[a,b]$ be the set of integers $a\le x\le b$. 
An {\em antimagic labeling} for $G$ is a bijection $f$ between $E(G)$ and $[1,m]$ such that all vertices sums are pairwise distinct,  where the {\em vertex sum} $\phi(v)$ of a vertex $v$ in $V(G)$ is the sum of the labels $f(e)$ for all edges $e$ incident to $v$.
We say that a graph $G$ is {\em antimagic} if there exists an antimagic labeling for $G$. 
It is clear that this kind of labeling does not exist if $G=K_2$. 
Hartsfield and Ringel~\cite{HR90} initiated the study of the antimagic problem of graphs. 
They gave the following two conjectures and confirmed the validity for some types of graphs. 

\begin{conjecture}{\rm\cite{HR90}}\label{conj}
Every connected graph $G\neq K_2$ is antimagic.  
\end{conjecture}

\begin{conjecture}{\rm\cite{HR90}}\label{conj2}
Every tree $T\neq K_2$ is antimagic.  
\end{conjecture}

Although many significant classes of graphs such as dense graphs~\cite{AKLRY04}, regular graphs~\cite{CLZ15,CLPZ16}, trees with at most one vertex of degree two~\cite{KLR09,LWZ14}, and caterpillars~\cite{LMST21} have been proved to be antimagic, the above conjectures remain widely open. 
Besides, various types of antimagic problems have been come up with and studied. 
In~\cite{G19}, a comprehensive dynamic survey of the related results and the latest developments is maintained. 
Notice that the assumption of connectedness in Conjecture~\ref{conj} cannot be dropped.  
There exist disconnected graphs which are not antimagic. 
In~\cite{CCLP21}, Chang, Chen, Pan, and the first author showed that for any graph $G$, the disjoint union of $G$ and sufficiently many $P_3$'s, a path on $3$ vertices, is not antimagic. Chaves, Le, Lin, Liu, and Shurman~\cite{CLLLS23} formulated the following parameter for graphs and studied it for many graphs. 

\begin{definition}{\rm\cite{CLLLS23}}
For a graph $G$,  let $\tau(G)$ be the maximum integer such that the disjoint union of $G$ and $c$ copies of $P_3$, is antimagic for $0\le c\le \tau(G)$. 
If $G$ is not antimagic, then $\tau(G)=-\infty$.
\end{definition}

Given a graph $G$, let $n$, $m$, $\ell$, and 
 $t$ be the number of vertices, edges, internal edges, and components of $G$ isomorphic to $P_3$, respectively. Define 
\[
\beta(G)=\min\{
\lfloor(3+2\sqrt{2})(m-n)+(1+\sqrt{2})(m+\tfrac{1}{2})\rfloor,  2m+5(\ell-t)+1\}.
\}
\]
Chaves {\em et al}.~\cite{CLLLS23} proved the following Theorem~\ref{theorem} and proposed several problems.  
\begin{theorem}[\rm Corollary 2,~\cite{CLLLS23}]\label{theorem}
For a graph $G$, $\tau(G)\le \beta(G)$.
\end{theorem}

\begin{conjecture}[\rm Conjecture 4,~\cite{CLLLS23}]\label{conjecture}
For a graph $G$, the union of $G$ and $c$ copies of $P_3$ is antimagic if and only if $c\le \tau(G)$.
\end{conjecture}

A {\em jellyfish} $J(C_k,r)$ is a graph obtained by attaching $r$ pendent edges to each vertex of the cycle $C_k$, In~\cite{CLLLS23}, Chaves {\em et al}. proved that $\tau(C_n)=\beta(C_n)$ for $3\le n\le 9$ and also $\tau(J(C_3,r))=\beta(J(C_3,r))$ for $r\ge 11$.
In addition, Conjecture~\ref{conjecture} holds for these graphs. 
A {\em star} $S_{n}$ is a tree with $n$ pendent edges incident to a common vertex. 
Much earlier, Chen, Huang, Lin, Shang, and Lee~\cite{CHLSL20} determined the value of $\tau(S_{n})$. 
They showed not only $\tau(S_n)=\beta(S_n)$, but also that $S_n$ satisfies Conjecture~\ref{conjecture}. 
A {\em double star} $S_{a, b}$ is a tree consisting of an internal edge with $a$ edges incident to one endpoint and $b$ edges incident to the other endpoint of the internal edge. 
Li~\cite{L19} proved that every {\em balanced double star} $S_{a,a}$ satisfies both $\tau(S_{a,a})=\beta(S_{a,a})$ and Conjecture~\ref{conjecture} in his master thesis.
The graphs $G$ for which $\tau(G)$ is known to date all satisfy $\tau(G)=\beta(G)$. Therefore, Cheves {\em et al.} asked the following question. 

\begin{question}[\rm Question 2,~\cite{CLLLS23}]\label{question}
Does there exist a graph $G$ without isolated vertices or $P_2$ as components with $\beta(G)\ge  0$ and $\tau(G) < \beta(G)$?
\end{question}

In this paper, we continue the study of the antimagicness of $S_{a,b}$ union the $P_3$'s.
We have completely determined $\tau(S_{a,b})$ for all $a$ and $b$, and also have proved that Conjecture~\ref{conjecture} holds for double stars by constructing all the antimagic labelings. 
Notice that Theorem~\ref{theorem} gives that  
\begin{equation}\label{upb0}
\tau(S_{a, b})\le 2m+6.
\end{equation}
for $m\ge 26$,  
and  
\begin{equation}\label{upb01}
\tau(S_{a, b})\le (1+\sqrt{2})(m-\tfrac{3}{2})-1
\end{equation}
for $m\le 25$,  where $m=a+b+1$ is the number of edges of $S_{a, b}$. 
It turns out that Inequalities (\ref{upb0}) and (\ref{upb01}) are strict for certain $a$ and $b$. 
As a consequence, we provide examples for which $\tau(S_{a,b})<\beta(S_{a,b})$, which answer Question~\ref{question}.
Our main result is the following theorem.

\begin{theorem}\label{main}
For the double star $S_{a, b}$ with $a\le b$ and $m=a+b+1$,  define
\[
\tau_0=
 \lfloor(1+\sqrt{2})(m-\tfrac{3}{2})\rfloor-1
\mbox{ and }
\tau_i=2m+i
\mbox{ for }
1\le i\le 6.
\]
Then we have 
\[
\tau(S_{a, b})=\left\{
\begin{array}{ll}
 \min\{\tau_0, \tau_1\},       &1\le a\le 2  \\
 \min\{\tau_0, \tau_{a-1}\},     &2\le a\le 6  \\
 \min\{\tau_0, \tau_5\},       & a=7\mbox{ and }b \le 21\\
 \min\{\tau_0, \tau_6\},       & a=7\mbox{ and }b \ge 22\mbox{ or } a \ge 8\\
\end{array}
\right.
\]
\end{theorem} 

Table~\ref{tab:upb} is the list of all $\tau(S_{a,b})$ and the numerical values of $\tau_0$ for small $m$. 
In each row, the number in the cell in the most right is the upper bound of $\tau(S_{a,b})$ obtained from Inequalities (\ref{upb0}) and (\ref{upb01}).
The gray cells show the improvement of the bounds for small $a$. 
We will derive the new upper bounds for $\tau(S_{a,b})$ in the next section. 
The construction methods of the antimagic labelings for $1\le c\le \tau(S_{a,b})$ will be given in Section 3. 
Particularly, some of our antimagic labelings have the property that all vertex sums are exactly 1 through the number of vertices. We will present these special cases and discuss the corresponding research in the last section.

\clearpage

\begin{table}[ht]
\begin{center}
    
\begin{tabular}{|c|c|c|c|c|c|c|c|c|c|c|c|c|c|c|}
\hline 
\diagbox{$m$}{$\tau(S_{a,b})$}{$a$}&  1 & 2 &  3 & 4  & 5 & 6 & 7 & 8 & 9 & 10 & 11 & 12 & 13 & $\ge $14 \\
\hline 3 & $\tau_0 $ &  &  &  &  &  &  &  &  &  &  &  &  &  \\
\hline 4 & $\tau_0 $ &  &  &  &  &  &  &  &  &  &  &  &  &  \\
\hline 5 & $\tau_0 $ & $\tau_0$ &  &  &  &  &  &  &  &  &  &  &  &  \\
\hline 6 & $\tau_0 $ & $\tau_0$ &  &  &  &  &  &  &  &  &  &  &  &  \\
\hline 7 & $\tau_0 $ & $\tau_0$ & $\tau_0$ &  &  &  &  &  &  &  &  &  &  &  \\
\hline 8 & $\tau_0 $ & $\tau_0$ & $\tau_0$ &  &  &  &  &  &  &  &  &  &  &  \\
\hline 9 & $\tau_0 $ & $\tau_0$ & $\tau_0$ & $\tau_0$ &  &  &  &  &  &  &  &  &  &  \\
\hline 10 & $\tau_0 $ & $\tau_0$ & $\tau_0$ & $\tau_0$ &  &  &  &  &  &  &  &  &  &  \\
\hline 11 & $\tau_0 $ & $\tau_0$ & $\tau_0$ & $\tau_0$ & $\tau_0$ &  &  &  &  &  &  &  &  &  \\
\hline $12$ & $\tau_0 $ & $\tau_0$ & $\tau_0$ & $\tau_0$ & $\tau_0$ &  &  &  &  &  &  &  &  &  \\
\hline $ 13$ & $\tau_0 $ & $\tau_0$ & $\tau_0$ & $\tau_0$ & $\tau_0$ & $\tau_0$ &  &  &  &  &  &  &  &  \\
\hline
\hline $14$ 
& $\tau_1^*$ & $\tau_1^*$ & $\tau_1^*$ & $\tau_1^*$ & $\tau_1^*$ & $\tau_1^*$ &&&&&&&&\\
\hline $15$ 
& $\tau_1^*$ & $\tau_1^*$ & $\tau_1^*$ & $\tau_1^*$ & $\tau_1^*$ & $\tau_1^*$ & $\tau_1^*$ &&&&&&&\\
\hline $16$ 
&\cellcolor{lightgray} $\tau_1$ &\cellcolor{lightgray} $\tau_1$ & $\tau_2^*$ & $\tau_2^*$ & $\tau_2^*$ & $\tau_2^*$ & $\tau_2^*$ &&&&&&&\\
\hline $17$ 
&\cellcolor{lightgray} $\tau_1$ &\cellcolor{lightgray} $\tau_1$ & $\tau_2^*$ & $\tau_2^*$ & $\tau_2^*$ & $\tau_2^*$ & $\tau_2^*$ & $\tau_2^*$ &&&&&&\\
\hline $18$ 
&\cellcolor{lightgray} $\tau_1$ &\cellcolor{lightgray} $\tau_1$ & $\tau_2^*$ & $\tau_2^*$ & $\tau_2^*$ & $\tau_2^*$ & $\tau_2^*$ & $\tau_2^*$ &&&&&&\\
\hline $19$ 
&\cellcolor{lightgray} $\tau_1$ &\cellcolor{lightgray} $\tau_1$ & \cellcolor{lightgray} $\tau_2$ & $\tau_3^*$ & $\tau_3^*$ & $\tau_3^*$ & $\tau_3^*$ & $\tau_3^*$ & $\tau_3^*$ &&&&&\\
\hline $20$ 
&\cellcolor{lightgray} $\tau_1$ &\cellcolor{lightgray} $\tau_1$ &\cellcolor{lightgray} $\tau_2$ & $\tau_3^*$ & $\tau_3^*$ & $\tau_3^*$ & $\tau_3^*$ & $\tau_3^*$ & $\tau_3^*$ &&&&&\\
\hline $21$ 
&\cellcolor{lightgray} $\tau_1$ &\cellcolor{lightgray} $\tau_1$ &\cellcolor{lightgray} $\tau_2$ &\cellcolor{lightgray} $\tau_3$ & $\tau_4^*$ & $\tau_4^*$ & $\tau_4^*$ & $\tau_4^*$ & $\tau_4^*$ & $\tau_4^*$ &&&&\\
\hline $22$ 
&\cellcolor{lightgray} $\tau_1$ &\cellcolor{lightgray} $\tau_1$ & \cellcolor{lightgray} $\tau_2$ &\cellcolor{lightgray} $\tau_3$ & $\tau_4^*$ & $\tau_4^*$ & $\tau_4^*$ & $\tau_4^*$ & $\tau_4^*$ & $\tau_4^*$ &&&&\\
\hline $23$ 
&\cellcolor{lightgray} $\tau_1$ &\cellcolor{lightgray} $\tau_1$ &\cellcolor{lightgray} $\tau_2$ &\cellcolor{lightgray} $\tau_3$ & $\tau_4^*$ & $\tau_4^*$ & $\tau_4^*$ & $\tau_4^*$ & $\tau_4^*$ & $\tau_4^*$ &$\tau_4^*$ &&&\\
\hline $24$ 
&\cellcolor{lightgray} $\tau_1$ &\cellcolor{lightgray} $\tau_1$ &\cellcolor{lightgray} $\tau_2$ &\cellcolor{lightgray} $\tau_3$ &\cellcolor{lightgray} $\tau_4$ & $\tau_5^*$ & $\tau_5^*$ & $\tau_5^*$ & $\tau_5^*$ & $\tau_5^*$ & $\tau_5^*$ &  &&\\
\hline $25$ 
&\cellcolor{lightgray} $\tau_1$ &\cellcolor{lightgray} $\tau_1$ &\cellcolor{lightgray} $\tau_2$ &\cellcolor{lightgray} $\tau_3$ &\cellcolor{lightgray} $\tau_4$ & $\tau_5^*$ & $\tau_5^*$ & $\tau_5^*$ & $\tau_5^*$ & $\tau_5^*$ & $\tau_5^*$ & $\tau_5^*$ &&\\
\hline
\hline $26$ 
&\cellcolor{lightgray} $\tau_1$ &\cellcolor{lightgray} $\tau_1$ &\cellcolor{lightgray} $\tau_2$ &\cellcolor{lightgray} $\tau_3$ &\cellcolor{lightgray} $\tau_4$ &\cellcolor{lightgray} $\tau_5$ &\cellcolor{lightgray} $\tau_5$ & $\tau_6$ & $\tau_6$ & $\tau_6$ & $\tau_6$ & $\tau_6$ &  &\\
\hline $27$ 
&\cellcolor{lightgray} $\tau_1$ &\cellcolor{lightgray} $\tau_1$ &\cellcolor{lightgray} $\tau_2$ &\cellcolor{lightgray} $\tau_3$ &\cellcolor{lightgray} $\tau_4$ &\cellcolor{lightgray} $\tau_5$ &\cellcolor{lightgray} $\tau_5$ & $\tau_6$ & $\tau_6$ & $\tau_6$ & $\tau_6$ & $\tau_6$ & $\tau_6$ &\\
\hline $28$ 
&\cellcolor{lightgray} $\tau_1$ &\cellcolor{lightgray} $\tau_1$ &\cellcolor{lightgray} $\tau_2$ &\cellcolor{lightgray} $\tau_3$ &\cellcolor{lightgray} $\tau_4$ &\cellcolor{lightgray} $\tau_5$ &\cellcolor{lightgray} $\tau_5$ & $\tau_6$ & $\tau_6$ & $\tau_6$ & $\tau_6$ & $\tau_6$ & $\tau_6$ &  \\
\hline $29$ 
&\cellcolor{lightgray} $\tau_1$ &\cellcolor{lightgray} $\tau_1$ &\cellcolor{lightgray} $\tau_2$ &\cellcolor{lightgray} $\tau_3$ &\cellcolor{lightgray} $\tau_4$ &\cellcolor{lightgray} $\tau_5$ &\cellcolor{lightgray} $\tau_5$ & $\tau_6$ & $\tau_6$ & $\tau_6$ & $\tau_6$ & $\tau_6$ & $\tau_6$ & $\tau_6$ \\
\hline $\ge 30$ 
&\cellcolor{lightgray} $\tau_1$ &\cellcolor{lightgray} $\tau_1$ &\cellcolor{lightgray} $\tau_2$ &\cellcolor{lightgray} $\tau_3$ &\cellcolor{lightgray} $\tau_4$ &\cellcolor{lightgray} $\tau_5$ & $\tau_6$ & $\tau_6$ & $\tau_6$ & $\tau_6$ & $\tau_6$ & $\tau_6$ & $\tau_6$ & $\tau_6$ \\
\hline
\end{tabular}
\end{center}
(1) $m=a+b+1$. (2) $\tau_i^*$ means $\tau_i=\tau_0$.\bigskip

The exact values of $\tau_0$ for $m\le 13$: 
\begin{center}
\begin{tabular}{|c|c|c|c|c|c|c|c|c|c|c|c|}
\hline
$m$ &3&4&5&6&7&8&9&10&11&12&13 \\
\hline $\tau_0$ & 2 & 5& 7& 9&12&14&17&19&21&24& 26\\
\hline
\end{tabular}    
\end{center}
\caption{$\tau(S_{a, b})$ for $1\le a\le b$.}
    \label{tab:upb}
\end{table}

\clearpage

\section{Upper bounds}

Throughout the paper, $G+H$ stands for the disjoint union of two graphs $G$ and $H$, and $cG$ is the disjoint union of $c$ copies of $G$. 
For $S_{a, b}+cP_3$, we use $u_1,u_2,\ldots, u_c$ to denote the internal vertices of the $P_3$'s and  $u_{c+1}$ and $u_{c+2}$ to denote the internal vertices in $S_{a,b}$ incident to $a$ and $b$ pendent edges, respectively. 
For $1\le i \le c$, let $E_i$ consist of the pendent edges incident to $u_i$. 
Moreover, $E_A$ and $E_B$ are the sets of pendent edges incident to $u_{c+1}$ and $u_{c+2}$, respectively, and $E_I=\{u_{c+1}u_{c+2}\}$.

\begin{figure}[ht]
    \centering
\begin{picture}(320, 55)
\textcolor{gray}{
\put(15,25){\oval(15,55)}
\put(55,25){\oval(15,55)}
\put(95,25){\oval(15,55)}
\put(150,25){\oval(15,55)}
\put(195,25){\oval(30,55)}
\put(295,25){\oval(30,55)}
\put(245,25){\oval(60,10)}
}
\put(15, 5){\circle*{4} }
\put(15, 25){\circle*{4} $u_1$}
\put(25,5){$E_1$}
\put(65,5){$E_2$}
\put(105,5){$E_3$}
\put(160,5){$E_c$}
\put(210,0){$E_A$}
\put(240,35){$E_I$}
\put(310,0){$E_B$}
\put(15, 45){\circle*{4} }
\put(15, 5){\line(0, 1){40}}
\put(55, 5){\circle*{4} }
\put(55, 25){\circle*{4} $u_2$}
\put(55, 45){\circle*{4} }
\put(55, 5){\line(0, 1){40}}
\put(95, 5){\circle*{4} }
\put(95, 25){\circle*{4} $u_3$}
\put(95, 45){\circle*{4} }
\put(95, 5){\line(0, 1){40}}
\put(150, 5){\circle*{4} }
\put(150, 25){\circle*{4} $u_c$}
\put(150, 45){\circle*{4} }
\put(150, 5){\line(0, 1){40}}
\put(120, 25){$\ldots$}
\put(190, 5){\circle*{4}}
\put(190, 12){$\vdots$}
\put(190, 30){\circle*{4}}
\put(190, 45){\circle*{4}}
\put(220, 25){\circle*{4}}
\put(220, 25){\line(-3, 2){30}}
\put(220, 25){\line(-3, -2){30}}
\put(220, 25){\line(-6, 1){30}}
\put(220, 25){\line(1, 0){50}}
\put(270, 25){\line(3, 2){30}}
\put(270, 25){\line(3, -2){30}}
\put(270, 25){\line(6, 1){30}}
\put(270, 25){\circle*{4}}
\put(300, 5){\circle*{4} }
\put(300, 12){$\vdots$}
\put(300, 30){\circle*{4} }
\put(300, 45){\circle*{4} }
\put(220, 12){$u_{c+1}$}
\put(250, 12){$u_{c+2}$}
\put(315, 45){$S_{a, b}$}
\end{picture}
    \caption{$S_{a, b}+cP_3$}
    \label{notatopm}
\end{figure}
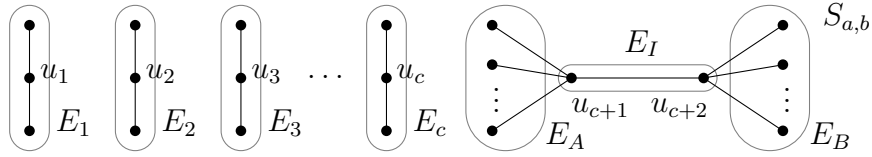

Denote $m=a+b+1$ the number of edges of $S_{a,b}$.
\begin{lemma}\label{upb2}
If $S_{1, b}+cP_3$ is antimagic, then $c\le 2m+1$.
\end{lemma}

\begin{proof} 
Let $f$ be an antimagic labeling for $S_{1,b}+cP_3$. 
We consider two possibilities for $\phi(u_i)$'s.
Suppose that $\phi(u_i)\le m+2c$ for some $i$. Note that there is an edge $e$ whose label $f(e)$ is equal to $\phi(u_i)$.   
We must have $e=u_{c+1}u_{c+2}$, the only internal edge; otherwise, the pendent vertex incident to $e$ and the vertex $u_i$ have the same vertex sum, a contradiction. 
Moreover, there is at most one $i$ with $\phi(u_i)\le m+2c$, and $1\le i\le c$ since $\phi(u_i)>f(u_{c+1}u_{c+2})$ for $i=c+1,c+2$. 
The other possibility is that $\phi(u_i)> m+2c$ for all $1\le i\le c+2$.

Suppose that $\phi(u_i)\le m+2c$ for some $i$. 
May assume that $i=1$. Then 
\begin{equation}\label{upb2_0}
    \sum_{i=2}^{c+1}\phi(u_i) \ge\sum_{i=1}^{c}(m+2c+i)=\tfrac{1}{2}c(2m+5c+1).
\end{equation}
Note that $\deg(u_{c+1})=2$. 
An upper bound for the left-hand side of Inequality (\ref{upb2_0}) is the sum of the largest $2c$ labels. Hence \[
\sum_{i=2}^{c+1}\phi(u_i)
\le\sum_{j=1}^{2c} (m+2c+1-j)
= c(2m+2c+1). 
\]
Observe that 
\begin{equation}\label{upb2_1}
c(2m+2c+1)\ge \tfrac{1}{2}c(2m+5c+1)
\Longrightarrow 
c\le 2m+1.
\end{equation}

Next suppose that $\phi(u_i)>m+2c$ for all $i$. 
Then 
\begin{equation}\label{upb2_3}
 \sum_{i=1}^{c+1}\phi(u_i) \ge\sum_{i=1}^{c+1}(m+2c+i)=\tfrac{1}{2}(c+1)(2m+5c+2),    
\end{equation}
and the left side of (\ref{upb2_3}) is at most 
\[
\sum_{j=1}^{2c+2} (m+2c+1-j)
=(c+1) (2m+2c-1). 
\]
Observe that 
\begin{equation}\label{upb2_2}
(c+1)(2m+2c-1) \ge  \tfrac{1}{2}(c+1)(2m+5c+2)
\Longrightarrow 
c\le 2m-4. 
\end{equation}
Thus we have $c\le 2m+1$ if $f$ is an antimagic labeling. 
\end{proof}

\begin{lemma}\label{upb3}
For $2\le a\le 6$,  if $S_{a, b}+cP_3$ is antimagic,  then 
\[c\le3a+2b+1=2m+a-1.\]
Moreover, if $b\le 21$ and $S_{7, b}+cP_3$ are antimagic, then $c\le 2b+21=2m+5$. 
\end{lemma}

\begin{proof}
As the proof of Lemma~\ref{upb2}, we consider an arbitrary antimagic labeling $f$ for $S_{a,b}+cP_3$ and discuss the two possibilities for $\phi(u_i)$'s.

Suppose $\phi(u_1)\le m+2c$. The lower bound for $\sum_{i=2}^{c+1}\phi(u_i)$ in (\ref{upb2_0}) remains valid.
Now that $u_{c+1}$ is incident to $a+1$ edges, we can bound $\sum_{i=2}^{c+1}\phi(u_i)$ above by the sum of the largest $a+2c-1$ labels. 
For convenience, we write the sum of these labels as follows:  
\[
\sum_{j=1}^{2c}(m+2c+1-j)+\sum_{j=1}^{a-1}(m+1-j)
=c(2m+2c+1)+\tfrac{1}{2}(a-1)(2m-a+2).
\]
Observe that 
\begin{eqnarray*}
&&c(2m+2c+1)+\tfrac{1}{2}(a-1)(2m-a+2)\ge \tfrac{1}{2}c(2m+5c+1)\\
&\Longleftrightarrow&
-[c-\tfrac{1}{2}(2m+1)]^2+\tfrac{1}{4}[(2m+1)^2+4(a-1)(2m-a+2)]\ge 0,  
\end{eqnarray*}
and the last inequality implies 
\begin{equation}\label{upb3_1}
c\le \tfrac{1}{2}(2m+1+\sqrt{(2m+1)^2+4(a-1)(2m-a+2)}).
\end{equation}

Assume that $\phi(u_i)>m+2c$ for all $i$.
Then Inequality (\ref{upb2_3}) is again a lower bound for $\sum_{i=1}^{c+1}\phi(u_i)$.
On the other hand, an upper bound for $\sum_{i=1}^{c+1}\phi(u_i)$ is the sum of the largest $a+2c+1$ labels, which is equal to  
\begin{eqnarray*}
\sum_{j=1}^{2c+a+1} (m+2c+1-j)
&=&\sum_{j=1}^{2c+2} (m+2c+1-j)+\sum_{j=1}^{a-1} (m-1-j)\\
&=& (c+1) (2m+2c-1)+\tfrac{1}{2}(a-1)(2m-a-2).  
\end{eqnarray*}
The two bounds lead to the following inequality which can be rearranged further:
\begin{eqnarray*}
&&(c+1) (2m+2c-1)+\tfrac{1}{2} (a-1)(2m-a-2)\ge 
\tfrac{1}{2}(c+1)(2m+5c+2)\\ 
&\Longleftrightarrow&
-[c-\tfrac{1}{2}(2m-5)]^2+\tfrac{1}{4}[(2m-3)^2+4(a-1)(2m-a-2)]\ge 0.  
\end{eqnarray*}
This implies 
\begin{equation}\label{upb3_2}
    c\le\tfrac{1}{2}(2m-5+\sqrt{(2m-3)^2+4(a-1)(2m-a-2)})
\end{equation}
Comparing (\ref{upb3_1}) and (\ref{upb3_2}),  we have  $c\le \tfrac{1}{2}(2m+1+\sqrt{(2m+1)^2+4(a-1)(2m-a+2)})$.
Moreover, Inequality (\ref{upb3_1}) can be further improved as 
\begin{equation}\label{upb3_3}
c< \tfrac{1}{2}(2m+1+\sqrt{(2m+2a-1)^2})=2m+a.    
\end{equation} 
because
\begin{eqnarray*}
    (2m+1)^2+4(a-1)(2m-a+2)
    &=&(2m+1)^2+4(2m+2)(a-1)-4(a-1)^2\\
    &<&(2m+1)^2+4(2m+2)(a-1)+4(a-1)^2\\
    &=&(2m+2a-1)^2.
\end{eqnarray*}
The upper bound for $c$ in (\ref{upb3_3}) is better than that in (\ref{upb0}) for $2\le a\le 6$.
We plug $a=7$ into (\ref{upb3_1}) and compare it with (\ref{upb0}). Then 
\[
\tfrac{1}{2}(2m+1+\sqrt{(2m+1)^2+4(a-1)(2m-a+2)})<2m+6
\]
holds for $m\le 29$, or equivalently for $b\le 21$.
\end{proof}

From the proofs of the previous lemmas, we see that if 
the label of the internal edge of $S_{a,b}$ is equal to the vertex sum of the internal vertex of some $P_3$, then we have a larger tolerance of the number of $P_3$'s.
Thus, our strategy for constructing antimagic labeling is to arrange the vertex sums of $u_i$'s, $1\le i\le c$, as small as possible but only one of them is smaller than the largest label, and the label equal to the smallest $\phi(u_i)$, $1\le i\le c$, will be assigned to the internal edge of $S_{a,b}$. Meanwhile, the vertex sums of $u_{c+1}$ and $u_{c+2}$ shall be arranged relatively large and distinct to other vertex sums. With this principle, we show the next proposition that is crucial for our constructions.

\begin{proposition}\label{labeling}
Let $a$, $b$ and $c$ be positive integers with $c\le \tau_0$. If $f$ is a bijection between $E(S_{a,b}+cP_3)$ and $[1,m+2c]$ such that 
\[\{\phi(u_1),\phi(u_2),\ldots,\phi(u_{c+1})\}=f(E_I)\cup[m+2c+1,m+3c],\] then $f$ is an antimagic labeling. 
\end{proposition}

\begin{proof}
Since the vertex sum of each pendent vertex is at most $m+2c$ and not equal to $f(u_{c+1}u_{c+2})$, it suffices to show $\phi(u_{c+2})\ge m+3c+1$. Observe that  
\begin{eqnarray*}
\phi({u_{c+2}})&=&\sum_{j=1}^{m+2c}j+f(E_I)-\sum_{i=1}^{c+1}\phi(u_{i})\\
&=&\tfrac{1}{2}(m+2c)(m+2c+1)-\tfrac{1}{2}c(2m+5c+1)\\
&=&\tfrac{1}{2}(m^2+2mc-c^2+m+c)\\
&\ge& m+3c+1
\end{eqnarray*}
holds when $c\le \lfloor\frac{1}{2}(2m-5+\sqrt{8m^2-24m+17})\rfloor$.
Meanwhile, 
\[\lfloor\tfrac{1}{2}(2m-5+\sqrt{8m^2-24m+17})\rfloor=\tau_0\] 
for every integer $m\ge 2$ can be verified with computers. So, the condition on $c$ guarantees that $\phi(u_{c+2})\ge m+3c+1$, and hence $f$ is an antimagic labeling. 
\end{proof}

\section{The antimagic labelings}

In the previous section, we determined the necessary condition on $c$ for which $S_{a,b}+cP_3$ is antimagic, which provides an upper bound for $\tau(S_{a,b})$. 
Next, we will construct the antimagic labeling for $S_{a,b}+cP_3$ with every $c$ not greater than these bounds. 
We will construct the labelings in terms of $c$.
More precisely, given $c$, we construct the antimagic labeling for $S_{a,b}+cP_3$ provided that $a$, $b$ and $c$ satisfy the conditions in Section 2. In the sequel, let  $k=|E(S_{a,b}+cP_3)|=a+b+2c+1=m+2c$.
For an edge labeling $f$, let $f(E_i)=\{f(e)\mid e\in E_i\}$. 

\begin{description}

\item{$\bullet$} $1\le c\le 2$

Define 
$f_c$ as follows:  
\[
f_1(E_i)=\left\{
\begin{array}{cc}
 \{1, k-1\},   & i=1  \\
 \{k\},     & i=I  \\
 {[2,a+1]},  &  i=A\\
 {[a+2, k]},    & i=B
\end{array}
\right.
\mbox{ and }
f_2(E_i)=\left\{
\begin{array}{cc}
 \{1, k-1\},     & i=1  \\
 \{3, k-2\},     & i=2  \\
 \{k\},     & i=I  \\
 \{2\}\cup [4,a+2],     & i=A  \\
 {[a+3,k-3]},     & i=B
\end{array}
\right.
\]
It is easy to verify that $f_1$ and $f_2$ are antimagic labelings for $S_{a,b}+P_3$ and $S_{a,b}+2P_3$, respectively.  

\item{$\bullet$} $3\le c\le 5$ (for $m\ge 4$)

For $c=3,4$, we first give the partial labelings for  $f_3$ and  $f_4$, then determine $f_c(E_A)$ and $f_c(E_B)$ according to $a$: 
\[
f_3(E_i)=\left\{
\begin{array}{cc}
 \{1, k-1\},     & i=1  \\
 \{3, k-2\},     & i=2  \\
 \{5, k-3\},     & i=3  \\
 \{k\},     & i=I  \\
\end{array}
\right.\mbox{and }
f_4(E_i)=\left\{
\begin{array}{cc}
 \{1, k-1\},     & i=1  \\
 \{3, k-2\},     & i=2  \\
 \{5, k-3\},     & i=3  \\
 \{7, k-4\},     & i=4  \\
 \{k\},     & i=I  \\
\end{array}
\right..
\]
For $a=1$, let $f_c(E_A)=\{4\}$ and $f_c(E_B)$ be the set of the remaining labels. If $a\ge 2$, let $f_c(E_A)$ and $f_c(E_B)$ be the sets of the smaller $a$ labels and the larger $b$ labels among the unused ones, respectively. Thus, $\phi(u_{c+1})\ge k+4$. 
Let us verify $\phi(u_{c+2})>\phi(u_{c+1})$.
This is obvious for $a\ge 2$.
For $a=1$, since we study $S_{a,b}+cP_3$ for $c\ge 3$ only when $b\ge 2$, we have $\{2,6\}\subseteq f_c(E_B)$ and the inequality holds.

If we followed the above patterns to label $S_{a,b}+5P_3$, then labels of the $P_3$'s would be $1$ and $k-1$, $3$ and $k-2$, $5$ and $k-3$, $7$ and $k-4$, $9$ and $k-5$. However, this does not work for $S_{1,2}$ because we have $9=k-5$ for $S_{1,2}+5P_3$. Instead, we define 
\[
f_5(E_i)=\left\{
\begin{array}{cc}
 \{1, 2\},     & i=1  \\
 \{4, k-3\},     & i=2  \\
 \{7, k-4\},     & i=3  \\
 \{6, k-2\},     & i=4  \\
 \{5, k\},     & i=5  \\
 \{3\},     & i=I  \\
\end{array}
\right.,
\]
and for $f_5(E_A)$ and $f_5(E_B)$, first assign $k-1$ to  $f_5(E_A)$, and then take turns to distribute labels $8,9,\ldots,k-5$ to $f_5(E_A)$ and $f_5(E_B)$ alternatively until $f_5(E_A)$ contains $a$ labels. 
Finally, put all the remaining labels in $f_5(E_B)$.
If $a=1$, then we have $\phi(u_1)=3$, $\phi(u_2)=k+1$, $\phi(u_3)=k+3$, $\phi(u_4)=k+4$, $\phi(u_5)=k+5$, $\phi(u_6)=k+2$, and $\phi(u_7)\ge k+6$. So $f_5$ is an antimagic labeling. 
When $a\ge 2$, if the above labeling is not antimagic, then 
\[3+(k-1)+8\le \phi(u_6)= \phi(u_7)=3+9+\cdots+(k-5).\] 
In this case, we exchange $8$ and $9$.
The new labeling gives $\phi(u_7)+2=\phi(u_6)\ge k+11$, which is an antimagic labeling.

\end{description}

Before continuing the work for $c\ge 6$, we introduce some properties.

\begin{lemma}\label{2subset}
Let $k\in\mathbb{N}$ and $p=p(k):= \lfloor \frac{2k+1}{5}\rfloor$. Then we can pick $2p$ distinct numbers $\alpha_1,\alpha_2,\ldots,\alpha_p$ and $\beta_1,\beta_2,\ldots,\beta_p$  from $[1,k]$
such that $\alpha_i+\beta_i=k+i$ for $1\le i\le p$.
\end{lemma}

\begin{proof} 
 The $\alpha_i$ and $\beta_i$ are given as follows, 
\begin{center}
$\alpha_i=\left\{
\begin{array}{rl}
p-i ,    &1\le i\le \lfloor \frac{p}{2}\rfloor  \\
 2p+\frac{1+(-1)^p}{2}-i,    &\lfloor \frac{p}{2}\rfloor< i\le p 
\end{array}
\right.$
 and 
$\beta_i=\left\{
\begin{array}{rl}
k-p+2i ,    &1\le i\le \lfloor \frac{p}{2}\rfloor  \\
 k-2p-\frac{1+(-1)^p}{2}+2i,    &\lfloor \frac{p}{2}\rfloor< i\le p 
\end{array}.
\right.$
\end{center}
It is straightforward to verify that  
all $\alpha_i$'s and $\beta_i$'s are distinct by the condition $p=\lfloor \frac{2k+1}{5}\rfloor$, and $\alpha_i+\beta_i=k+i$ for $1\le i\le p$. 
Observe that the minimum of all $\alpha_i$'s and $\beta_i$'s is $\alpha_{\lfloor \frac{p}{2}\rfloor}=\lfloor\frac{p+1}{2}\rfloor$.
Also, $\{\beta_i\mid 1\le i\le p\}=[k-p+1,k]$. 
\end{proof}
Indeed, Lemma~\ref{2subset} is the best possible for finding pairs whose sums are distinct and larger than $k$. 
Recall that $k=m+2c$. When $c\le 2m+1$,
\begin{equation}\label{le2m+1}
p(k)=\lfloor\tfrac{2k+1}{5}\rfloor=\lfloor\tfrac{2m+4c+1}{5}\rfloor\ge c,  
\end{equation}
and when $2m+2\le c\le 2m+6$,  
\begin{equation}\label{2m+2to6}
p(k)=\lfloor\tfrac{2k+1}{5}\rfloor=\lfloor\tfrac{2m+4c+1}{5}\rfloor= c-1.   
\end{equation}

Generally speaking, we will use $\alpha_i$ and $\beta_i$ to label the edges of each $P_3$ for $i=1,2,\ldots$, accordingly.  
However, for convenience or due to deficiency of the number of pairs (when $2m+2\le c\le 2m+6$), we will find another pair to label one $P_3$. 

The following are some notations which will be used in the subsequent proofs.
The sum of all elements in a set $X$ is denoted by $\|X\|$. 
A set $X$ is called an {\em $(s,t)$-set} if $|X|=s$ and $\|X\|=t$. 
For $|X|\ge n$, $X_{\min}^{(n)}$ (resp., $X_{\max}^{(n)}$) refers to the subset of $X$ consisting of the smallest (resp., largest) $n$ elements in $X$. 
In particular, $\|X_{\min}^{(1)}\|$ and $\|X_{\max}^{(1)}\|$ are the minimum and maximum elements in $X$.

\begin{lemma}{\rm \cite{L19}}\label{stset}
If $Z$ is a set of consecutive integers such that $\|Z_{\min}^{(s)}\|\le t\le \|Z_{\max}^{(s)}\|$, then there exists an $(s,t)$-set $Y\subseteq Z$.
\end{lemma}

\begin{proof}
We prove by induction. If $t=\|Z_{\min}^{(s)}\|$, then $Z_{\min}^{(s)}$ is the $(s,t)$-set. 
Suppose that $\|Z_{\min}^{(s)}\|\le t\le \|Z_{\max}^{(s)}\|$ and $Y=\{z_1,z_2,\ldots,z_s\}\subseteq Z$ is an $(s,t)$-set. If we can replace some $z_i$ by $z_i+1\in Z-Y$, then $Y-\{z_i\}\cup\{z_i+1\}\subseteq Z$ is an $(s,t+1)$-set. Once we cannot do the replacement, it means that $Y$ already contains the largest $s$ integers in $Z$ and $t=\|Z_{\max}^{(s)}\|$. 
\end{proof}

\begin{description}

\item{$\bullet$} $6\le c\le 2m+1$ (for $m\ge 5$)

 By Lemma~\ref{2subset} and Inequality~(\ref{le2m+1}), we can take $c$ pairs of $\alpha_i,\beta_i\in [k]$ with $\alpha_i+\beta_i=k+i$ for $1\le i \le c$.  
Let $\alpha_{i_0}=\min_{1\le i\le c}\alpha_i$. 
The conditions on $c$ and $m$ imply $k\ge 17$. So  $p(k)=\lfloor \frac{2k+1}{5}\rfloor\ge 7$ and  $\alpha_{i_0}\ge \alpha_{\lfloor\frac{p}{2}\rfloor}\ge 4$. 

We consider two cases $a=1$ and $a\ge 2$. For $a=1$, 
define 
\[
f_c(E_i)=
\left\{
\begin{array}{cc}
\{\alpha_i,\beta_i\},     &  1\le i\le c,i\neq i_0 \\
\{1,\alpha_{i_0}-1\},     & i=i_0 \\
\{\alpha_{i_0}\}, & i=I \\
\{\beta_{i_0}\}, & i=A
\end{array}
\right.
\]
and let $f_c(E_B)$ consist of the remaining labels. 
The labeling $f_c$ satisfies Proposition~\ref{labeling}, so $S_{1,b}+cP_3$ is antimagic.

Suppose $a\ge 2$. If $p\ge c+1$, then 
let $X:=[4,k]-\cup_{i=1}^{c+1}\{\alpha_i,\beta_i\}$ 
and define 
\[
f_c(E_i)=
\left\{
\begin{array}{cc}
\{\alpha_i,\beta_i\},     &  1\le i\le c-1 \\
\{1,2\},     & i=c \\
\{3\}, & i=I \\
\{\alpha_c,\beta_c\} \cup X_{\min}^{(a-2)}, & i=A\\
\{\alpha_{c+1},\beta_{c+1}\}\cup X_{\max}^{(b-2)}, &i=B
\end{array}
\right..
\]
Then $f_c$ is an antimagic labeling by direct comparisons. 

Suppose $p=c$. Let  $X:=[4,k]-\cup_{i=1}^c\{\alpha_i,\beta_i\}$. Note that $ X_{\min}^{(a-1)}\cap X_{\max}^{(1)}=\varnothing$ since $|X|=a+b-2\ge a$. Also, at least one of the following inequalities holds. 
\begin{description}
\item{(1)} $3+\|X_{\max}^{(1)}\|+\|X_{\min}^{(a-2)}\| \le k$.

\item{(2)} $k< 3+\|X_{\max}^{(1)}\|+\|X_{\min}^{(a-2)}\|$ and $ 3+\lambda+\|X_{\min}^{(a-1)}\|\le k+c$. 

\item{(3)} $k+c<  3+\|X_{\max}^{(1)}\|+\|X_{\min}^{(a-1)}\|$. 

\end{description}

If (1) is true, then $3+\|X_{\max}^{(1)}\|+\|X_{\min}^{(a-2)}\| \in \{\alpha_{i_1},\beta_{i_1}\}$ for some $1\le i_1\le c$  since $\|X_{\max}^{(1)}\|<  3+\|X_{\max}^{(1)}\|+\|X_{\min}^{(a-2)}\| \le k$. 
Define 
\[
f_c(E_i)=
\left\{
\begin{array}{cc}
\{\alpha_i,\beta_i\},     &  1\le i\le c,i\neq i_1 \\
\{1,2\},     & i=i_1 \\
\{3\}, & i=I \\
\{k+i_1-3-\|X_{\max}^{(1)}\|-\|X_{\min}^{(a-2)}\|\}\cup X_{\max}^{(1)}\cup X_{\min}^{(a-2)}, & i=A
\end{array}
\right.
\]
and let $f_c(E_B)$ be the set of the remaining labels. 
Thus, $\{\phi(u_1),\phi(u_2),\ldots,\phi(u_{c+1})\}=\{3\}\cup[k+1,k+c]$.
By Proposition~\ref{labeling}, $f_c$ is an antimagic labeling. 

Next suppose that (2) holds. Then $3+\|X_{\max}^{(1)}\|+\|X_{\min}^{(a-1)}\|=k+i_2$ for some $1\le i_2\le c$. Then we define 
\[
f_c(E_i)=
\left\{
\begin{array}{cc}
\{\alpha_i,\beta_i\}     &  1\le i\le c,i\neq i_2 \\
\{1,2\}     & i=i_2 \\
\{3\} & i=I \\
X_{\max}^{(1)}\cup X_{\min}^{(a-1)} & i=A
\end{array}
\right.
\]
and let $f_c(E_B)$ be the set of the remaining labels. 
Again, by Proposition~\ref{labeling}, $f_c$ is an antimagic labeling. 

Finally, for (3), we define 
\[
f_c(E_i)=
\left\{
\begin{array}{cc}
\{\alpha_i,\beta_i\}     &  1\le i\le c-1 \\
\{1,2\}     & i=c \\
\{3\} & i=I \\
X_{\max}^{(1)}\cup X_{\min}^{(a-1)} & i=A
\end{array}
\right.
\]
and let $f_c(E_B)$ be the set of the remaining labels. Since $\phi(u_{c+1})> k+c\ge\phi(u_i)$ for $1\le i\le c$, it suffices to show $\phi(u_{c+2})>\phi(u_{c+1})$. 
Note that $c=p=\lfloor\frac{2k+1}{5}\rfloor> \frac{2k-4}{5}=\frac{2m+4c-4}{5}$. 
So $c>2m-4$ and $\alpha_c+\beta_c=k+c>k+2m-4> 2k-2c\ge 2\|X_{\max}^{(1)}\|$.  
Since every element in $f_c(E_B)$ other than $\alpha_c$ is greater than any element in $X_{\min}^{(a-1)}$, we have 
\[\|f_c(E_B)\|>\|f_c(E_B)-\{\alpha_c,\beta_c\}\|+2\|X_{\max}^{(1)}\|>\|X_{\min}^{(a-1)}\|+\|X_{\max}^{(1)}\|=\|f_c(E_A)\|.
\]
As a consequence,  $\phi(u_{c+2})>\phi(u_{c+1})$ and $f_c$ is an antimagic labeling.

\end{description}

For $2m+2\le c\le 2m+6$, by Equation~(\ref{2m+2to6}), we have only $c-1$ pairs of $\alpha_i$ and $\beta_i$ to label the $P_3$'s. Let $X=[1,k]-\cup_{i=1}^{c-1}\{\alpha_i,\beta_i\}$. 
In each case below, we will find an $(a+1)$-set $W\subseteq X$ such that 
$\|W\|=k+c$ or $\phi(u_{c+2})>\phi(u_{c+1})=\|W\|>k+c$.
Then define the labeling as:
\[
f_c(E_i)=
\left\{
\begin{array}{cc}
\{\alpha_i,\beta_i\}     &  1\le i\le c-1 \\
\{1,\|W_{\min}^{(1)}\|-1\}     & i=c \\
W_{\min}^{(1)} & i=I \\
W-W_{\min}^{(1)} & i=A\\
X-W-\{1,\|W_{\min}^{(1)}\|-1\},&i=B
\end{array}
\right.
\]
So, by Proposition~\ref{labeling} or by direct comparisons, we can see that $f_c$ is an antimagic labeling. 

\begin{description}
\item{$\bullet$} $c=2m+2$ (for $a\ge 3$ and  $m\ge 16$)

We have $k+c=7m+6$ and $X=[1,m]\cup\{3m+2,3m+3\}$. Let $Z=[3,m]$. 

If $3\le a\le 4$, then
\[
\|Z_{\min}^{(a-1)}\|\le 12<m+1<2m-1\le\|Z_{\max}^{(a-1)}\|.
\]
By Lemma~\ref{stset}, there is an $(a-1,m+1)$-set $Y\subseteq Z$. 
Thus $W=Y\cup\{3m+2,3m+3\}$ is the desired  $(a+1,7m+6)$ set.

When $5\le a\le 9$, 
\[
\|Z_{\min}^{(a)}\|\le 63<4m+4<5m-10\le\|Z_{\max}^{(a)}\|.
\]
By Lemma~\ref{stset}, there is an $(a,4m+4)$-set $Y\subseteq Z$. Take $W=Y\cup\{3m+2\}$. 

For $a\ge 10$, since $a\le b$, we have $m=a+b+1\ge 21$ and 
\[
7m+6<10m-45\le\|Z_{\max}^{(a+1)}\|.
\]
If $\|Z_{\min}^{(a+1)}\|\le  7m+6$, then there is an $(a+1,7m+6)$-set $Y\subseteq Z$. 
Otherwise, $\|Z_{\min}^{(a+1)}\|>  7m+6$. 
Take $W=Y$ in the former case and $W=Z_{\min}^{(a+1)}$ in the latter case.

\item{$\bullet$} $c=2m+3$ (for $a\ge 4$ and $m\ge 19$)

Then $k+c=7m+9$ and $X=[1,m]\cup\{2m+2,3m+4\}$. Let $Z=[3,m]$. 

If $4\le a\le 7$, then
\[
\|Z_{\min}^{(a-1)}\|\le 33<2m+3<3m-3\le\|Z_{\max}^{(a-1)}\|.
\]
By Lemma~\ref{stset}, there is an $(a-1,2m+3)$-set $Y\subseteq Z$. 
Thus, $W=Y\cup\{2m+2,3m+4\}$ is the $(a+1,7m+9)$-set. 

When $8\le a\le 12$, 
\[
\|Z_{\min}^{(a)}\|\le 102\le 5m+7<8m-28\le\|Z_{\max}^{(a)}\|.
\]
By Lemma~\ref{stset}, there is an $(a,5m+7)$-set $Y\subseteq Z$. Let  $W=Y\cup\{2m+2\}$. 

For $a\ge 13$, 
\[
7m+9<14m-91\le\|Z_{\max}^{(a+1)}\|.
\]
If $\|Z_{\min}^{(a+1)}\|\le  7m+9$, then take an $(a+1,7m+9)$-set $Y\subseteq Z$ as $W$. 
Otherwise, $\|Z_{\min}^{(a+1)}\|> 7m+9$  
and take $W=Z_{\min}^{(a+1)}$. 

\item{$\bullet$} $c=2m+4$ (for $a\ge 5$ and $m\ge 21$)

We have $k+c=7m+12$ and $X=[1,m+1]\cup\{3m+5\}$. Let $Z=[3,m+1]$. 

If $5\le a\le 11$, then
\[
\|Z_{\min}^{(a)}\|\le 88<4m+7<5m-5\le\|Z_{\max}^{(a)}\|.
\]
By Lemma~\ref{stset}, there is an $(a,4m+7)$-set $Y\subseteq Z$. 
Thus, take $W=Y\cup\{3m+5\}$.

For $a\ge 12$, 
\[
7m+12<13m-65\le\|Z_{\max}^{(a+1)}\|.
\]
If $\|Z_{\min}^{(a+1)}\|\le  7m+12$, then take an $(a+1,7m+12)$-set $Y\subseteq Z$ as $W$. 
Otherwise, $\|Z_{\min}^{(a+1)}\|> 7m+12$ and  let $W=Z_{\min}^{(a+1)}$.

\item{$\bullet$} $c=2m+5$ (for $a\ge 6$ and $m\ge 24$)

Then $k+c=7m+15$ and $X=[1,m+1]\cup\{2m+4\}$. Let $Z=[3,m+1]$. 

If $6\le a\le 13$, then
\[
\|Z_{\min}^{(a)}\|\le 117<5m+11<6m-9\le\|Z_{\max}^{(a)}\|.
\]
By Lemma~\ref{stset}, there is an $(a,5m+11)$-set $Y\subseteq Z$. 
Take $W=Y\cup\{2m+4\}$.

For $a\ge 14$, 
\[
7m+15<15m-97\le\|Z_{\max}^{(a+1)}\|.
\]
If $\|Z_{\min}^{(a+1)}\|\le  7m+15$, take an $(a+1,7m+15)$-set $Y\subseteq Z$ as $W$. 
Otherwise, $\|Z_{\min}^{(a+1)}\|> 7m+15$ and let $W=Z_{\min}^{(a+1)}$.

\item{$\bullet$} $c=2m+6$ (for $a= 7$ and $m\ge 30$ or for $a\ge 8$ and $m\ge 26$)

We have $k+c=7m+15$ and $X=[1,m+2]$. Let $Z=[3,m+2]$. 

For $a=7$ and $m\ge 30$,  
\[
7m+18\le 8m-12=\|Z_{\max}^{(8)}\|.
\]
When $\|Z_{\min}^{(8)}\|\le  7m+18$, take an $(8,7m+18)$-set $Y\subseteq Z$ as $W$. Otherwise $\|Z_{\min}^{(8)}\|> 7m+18$ and let $W=Z_{\min}^{(8)}$.

For $a\ge 8$ and $m\ge 26$, 
\[
7m+18<9m-18\le \|Z_{\max}^{(a+1)}\|.
\]
As before, if $\|Z_{\min}^{(a+1)}\|\le  7m+18$, take an $(a+1,7m+18)$-set $Y\subseteq Z$ as $W$. Otherwise,  $\|Z_{\min}^{(a+1)}\|> 7m+18$ and let $W=Z_{\min}^{(a+1)}$.

\end{description}

\section{$(a,d)$-antimagic labelings}

An {\em $(a,d)$-antimagic labeling} for a graph $G$ is an antimagic labeling such that the vertex sums of the vertices form an arithmetic progression with $a$ as the initial term and $d$ as the common difference. 
This concept was introduced by Bodendiek and Walther in 1993~\cite{BW93}. 
A necessary condition for which a graph $G$ on $n$ vertices and $m$ edge has an $(a,d)$-antimagic labeling is 
\begin{equation}\label{ad}
2an+n(n-1)d=2m(m+1),    
\end{equation}
given by Bodendiek and Walther~\cite{BW98}. 
They also showed that the paths $P_{2k+1}$ have  the $(k,1)$-antimagic labeling and also the cycles $C_{2k+1}$ have the $(k+2,1)$-antimagic labelings for $k\ge 1$. 
For more results on $(a,1)$-antimagic labelings, see~\cite{B00,BFWZ15, BH98,I16,IS07,JISAS13}.
Note that by Equation (\ref{ad}), if a graph $G$ on $n$ vertices and $m$ edges has a $(1,1)$-antimagic labeling, then $m$ and $n$ satisfy the {\em negative Pell equation} 
\begin{equation}\label{pell}
(2n+1)^2-2(2m+1)^2=-1.
\end{equation}
There are infinitely many solutions for Equation~(\ref{pell}), 
and the smallest positive $n$ is $n=3$ with $m=2$. The only graph on $3$ vertices and $2$ edges is $P_3$, and an antimagic labeling for $P_3$ induces the vertex sums exactly 1, 2, and 3.
Moreover, by the fact that the number of edges of a connected graph is at least the number of the vertices of it minus one, we can conclude that a connected graph has a $(1,1)$-antimagic labeling only if $n\le 3$. However, this is not the case for disconnected graphs. 
Let us see some examples. 
The second smallest $n$ that satisfies the equation is $n=20$ with $m=14$. 
We observed that $S_{1,2}+5P_3$ satisfies the conditions of $m$ and $n$, and our labeling in Section 3 is a $(1,1)$-antimagic labeling for $S_{1,2}+5P_3$. In addition, we also found out other $(1,1)$-antimagic labelings for graphs on $20$ vertices and $14$ edges. 
These graphs are $P_5+5P_3$, $S_4+5P_3$, $2P_4+4P_3$, $2S_3+4P_3$, and $P_4+S_3+4P_3$.  
We show that these are all the graphs on 20 vertices and 14 edges having this kind of labeling and present a $(1,1)$-antimagic labeling for each graph in Figure~\ref{n20m14}. 
Assume that $G$ is a graph on 20 vertices and 14 edges, and $G$ contains $c$ components.
Let $n_i$ and $m_i$ be the numbers of vertices and edges of the $i$th component for $1\le i\le c$. Then we have 
\[20=n_1+n_2+\cdots +n_c\le (m_1+1)+(m_2+1)+\cdots+(m_c+1)=14+c\]
and hence $c\ge 6$. In addition, each component contains at least 3 vertices, so $n_i\ge 3$. 
Consequently, $G$ consists of 6 components which are 5 $P_3$'s and one tree on 5 vertices or 4 $P_3$'s and two trees on 4 vertices. 

We provide another class of graphs that have the $(1,1)$-antimagic labelings. 
The third smallest $n$ is $n=119$ with $m=84$. 
For $S_{a,b}+ 34P_3$ with $a+b=15$ and $3\le a \le b$, it has the required numbers of vertices and edges, 
and our labelings in Section 3 for $S_{a,b}+ 34P_3$ give $\{\phi(u_1),\phi(u_2),\ldots,\phi(u_{35})\}=f(E_{I})\cup[85,118]$. Thus, the vertex sums of all vertices except $u_{36}$ in $S_{a,b}+34P_3$ are $1,2,\ldots,118$, which forces $\phi(u_{36})=119$ by double counting the labels of the edges and the vertex sums of the vertices.

For the next $n$, $n=696$ with $m=492$, 
if $S_{a,b}+cP_3$ satisfies  $|V(S_{a,b}+cP_3)|=696$ and  $|E(S_{a,b}+cP_3)|=492$, then 
$|E(S_{a,b})|=86$ and $c=203$. 
However, in this case $S_{a,b}+203P_3$ cannot be antimagic since $203>178=2|E(S_{a,b})|+6$.
Nonetheless, it might be possible to find a graph $G$ with $|E(G)|=86$, $\tau(G)\ge 203$, and a $(1,1)$-antimagic labeling for $G$. 

If there is a graph on $n$ vertices and $m$ edges such that $m$ and $n$ satisfy Equation~(\ref{pell}), and contains no isolated vertices and edges, then it must contain several components isomorphic to $P_3$. 
This is because if every component of the graph contains at least 3 edges, then one can deduce $m/n \ge 3/4$, which will lead to a contradiction when plugged into Equation~(\ref{pell}).  
As a consequence, a $(1,1)$-antimagic labeling can exist only for graphs of the form $G+cP_3$. 
Actually, if a graph of the form $G+cP_3$ satisfies Equation~(\ref{pell}), and is antimagic, then the antimagic labeling must be a $(1,1)$-antimagic labeling by double counting the labels and the vertex sums. 
This strengthens the motivation for the study of the parameter $\tau(G)$ of general graphs.

In the end, we propose the following conjecture.

\begin{conjecture} 
Let $m$ and $n$ be positive integers that satisfy $(2n+1)^2-2(2m+1)^2=-1$. Then there exist a graph $G$ and an integer $c$ such that 
$|V(G+cP_3)|=n$, $|E(G+cP_3)|=m$ and $G+cP_3$ has a $(1,1)$-antimagic labeling. 
\end{conjecture}

\begin{figure}[ht]
$S_{1, 2}+5P_3$
 
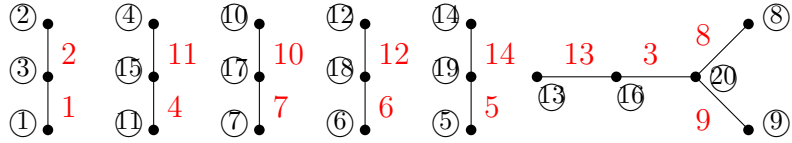
\begin{center}
\begin{picture}(300, 45)
\textcolor{red}{
\put(20,10){1}
\put(20,30){2}
\put(60,10){4}
\put(60,30){11}
\put(100,10){7}
\put(100,30){10}
\put(140,10){6}
\put(140,30){12}
\put(180,10){5}
\put(180,30){14}
\put(210,30){13}
\put(240,30){3}
\put(260,37){8}
\put(260,5){9}
}
\put(0, 5){\footnotesize\textcircled{1}}
\put(0, 25){\footnotesize\textcircled{3}}
\put(0, 45){\footnotesize\textcircled{2}}
\put(15, 5){\circle*{4} }
\put(15, 25){\circle*{4}}
\put(15, 45){\circle*{4} }
\put(15, 5){\line(0, 1){40}}
\put(55, 5){\circle*{4} }
\put(40, 5){\footnotesize\textcircled{11}}
\put(40, 25){\footnotesize\textcircled{15}}
\put(40, 45){\footnotesize\textcircled{4}}
\put(55, 25){\circle*{4}}
\put(55, 45){\circle*{4} }
\put(55, 5){\line(0, 1){40}}
\put(80, 5){\footnotesize\textcircled{7}}
\put(80, 25){\footnotesize\textcircled{17}}
\put(80, 45){\footnotesize\textcircled{10}}
\put(95, 5){\circle*{4} }
\put(95, 25){\circle*{4} }
\put(95, 45){\circle*{4} }
\put(95, 5){\line(0, 1){40}}
\put(120, 5){\footnotesize\textcircled{6}}
\put(120, 25){\footnotesize\textcircled{18}}
\put(120, 45){\footnotesize\textcircled{12}}
\put(135, 5){\circle*{4} }
\put(135, 25){\circle*{4} }
\put(135, 45){\circle*{4} }
\put(135, 5){\line(0, 1){40}}
\put(160, 5){\footnotesize\textcircled{5}}
\put(160, 25){\footnotesize\textcircled{19}}
\put(160, 45){\footnotesize\textcircled{14}}
\put(175, 5){\circle*{4} }
\put(175, 25){\circle*{4} }
\put(175, 45){\circle*{4} }
\put(175, 5){\line(0, 1){40}}
\put(285, 5){\footnotesize\textcircled{9}}
\put(285, 45){\footnotesize\textcircled{8}}
\put(200, 15){\footnotesize\textcircled{13}}
\put(230, 15){\footnotesize\textcircled{16}}
\put(265, 22){\footnotesize\textcircled{20}}
\put(200, 25){\circle*{4}}
\put(230, 25){\circle*{4}}
\put(260, 25){\circle*{4}}
\put(280, 5){\circle*{4}}
\put(280, 45){\circle*{4}}
\put(200, 25){\line(1, 0){60}}
\put(260, 25){\line(1, 1){20}}
\put(260, 25){\line(1, -1){20}}
\end{picture}
\end{center}

$P_5+5P_3$

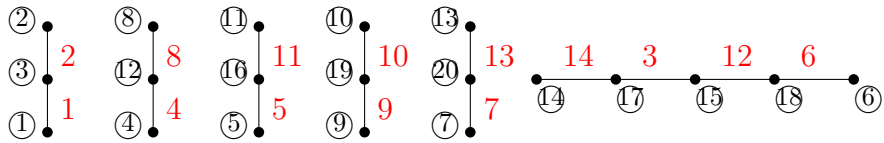
\begin{center}
\begin{picture}(330, 45)
\textcolor{red}{
\put(20,10){1}
\put(20,30){2}
\put(60,10){4}
\put(60,30){8}
\put(100,10){5}
\put(100,30){11}
\put(140,10){9}
\put(140,30){10}
\put(180,10){7}
\put(180,30){13}
\put(210,30){14}
\put(240,30){3}
\put(270,30){12}
\put(300,30){6}
}
\put(0, 5){\footnotesize\textcircled{1}}
\put(0, 25){\footnotesize\textcircled{3}}
\put(0, 45){\footnotesize\textcircled{2}}
\put(40, 5){\footnotesize\textcircled{4}}
\put(40, 25){\footnotesize\textcircled{12}}
\put(40, 45){\footnotesize\textcircled{8}}
\put(80, 5){\footnotesize\textcircled{5}}
\put(80, 25){\footnotesize\textcircled{16}}
\put(80, 45){\footnotesize\textcircled{11}}
\put(120, 5){\footnotesize\textcircled{9}}
\put(120, 25){\footnotesize\textcircled{19}}
\put(120, 45){\footnotesize\textcircled{10}}
\put(160, 5){\footnotesize\textcircled{7}}
\put(160, 25){\footnotesize\textcircled{20}}
\put(160, 45){\footnotesize\textcircled{13}}
\put(200, 15){\footnotesize\textcircled{14}}
\put(230, 15){\footnotesize\textcircled{17}}
\put(260, 15){\footnotesize\textcircled{15}}
\put(290, 15){\footnotesize\textcircled{18}}
\put(320, 15){\footnotesize\textcircled{6}}
\put(15, 5){\circle*{4} }
\put(15, 25){\circle*{4}}
\put(15, 45){\circle*{4} }
\put(15, 5){\line(0, 1){40}}
\put(55, 5){\circle*{4} }
\put(55, 25){\circle*{4}}
\put(55, 45){\circle*{4} }
\put(55, 5){\line(0, 1){40}}
\put(95, 5){\circle*{4} }
\put(95, 25){\circle*{4} }
\put(95, 45){\circle*{4} }
\put(95, 5){\line(0, 1){40}}
\put(135, 5){\circle*{4} }
\put(135, 25){\circle*{4} }
\put(135, 45){\circle*{4} }
\put(135, 5){\line(0, 1){40}}
\put(175, 5){\circle*{4} }
\put(175, 25){\circle*{4} }
\put(175, 45){\circle*{4} }
\put(175, 5){\line(0, 1){40}}
\put(200, 25){\circle*{4}}
\put(230, 25){\circle*{4}}
\put(260, 25){\circle*{4}}
\put(290, 25){\circle*{4}}
\put(320, 25){\circle*{4}}
\put(200, 25){\line(1, 0){120}}
\end{picture}
\end{center}

$S_4+5P_3$

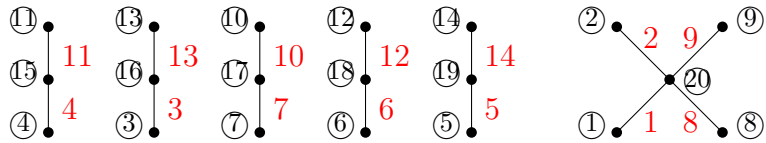
\begin{center}
\begin{picture}(280, 45)
\textcolor{red}{
\put(20,10){4}
\put(20,30){11}
\put(60,10){3}
\put(60,30){13}
\put(100,10){7}
\put(100,30){10}
\put(140,10){6}
\put(140,30){12}
\put(180,10){5}
\put(180,30){14}
\put(240,5){1}
\put(240,37){2}
\put(255,5){8}
\put(255,37){9}
}
\put(0, 5){\footnotesize\textcircled{4}}
\put(0, 25){\footnotesize\textcircled{15}}
\put(0, 45){\footnotesize\textcircled{11}}
\put(40, 5){\footnotesize\textcircled{3}}
\put(40, 25){\footnotesize\textcircled{16}}
\put(40, 45){\footnotesize\textcircled{13}}
\put(80, 5){\footnotesize\textcircled{7}}
\put(80, 25){\footnotesize\textcircled{17}}
\put(80, 45){\footnotesize\textcircled{10}}
\put(120, 5){\footnotesize\textcircled{6}}
\put(120, 25){\footnotesize\textcircled{18}}
\put(120, 45){\footnotesize\textcircled{12}}
\put(160, 5){\footnotesize\textcircled{5}}
\put(160, 25){\footnotesize\textcircled{19}}
\put(160, 45){\footnotesize\textcircled{14}}
\put(215, 5){\footnotesize\textcircled{1}}
\put(215, 45){\footnotesize\textcircled{2}}
\put(255, 22){\footnotesize\textcircled{20}}
\put(275, 5){\footnotesize\textcircled{8}}
\put(275, 45){\footnotesize\textcircled{9}}
\put(15, 5){\circle*{4} }
\put(15, 25){\circle*{4}}
\put(15, 45){\circle*{4} }
\put(15, 5){\line(0, 1){40}}
\put(55, 5){\circle*{4} }
\put(55, 25){\circle*{4}}
\put(55, 45){\circle*{4} }
\put(55, 5){\line(0, 1){40}}
\put(95, 5){\circle*{4} }
\put(95, 25){\circle*{4} }
\put(95, 45){\circle*{4} }
\put(95, 5){\line(0, 1){40}}
\put(135, 5){\circle*{4} }
\put(135, 25){\circle*{4} }
\put(135, 45){\circle*{4} }
\put(135, 5){\line(0, 1){40}}
\put(175, 5){\circle*{4} }
\put(175, 25){\circle*{4} }
\put(175, 45){\circle*{4} }
\put(175, 5){\line(0, 1){40}}
\put(230, 5){\line(1, 1){40}}
\put(230, 45){\line(1, -1){40}}
\put(230, 5){\circle*{4}}
\put(230, 45){\circle*{4}}
\put(250, 25){\circle*{4}}
\put(270, 5){\circle*{4}}
\put(270, 45){\circle*{4}}
\end{picture}
\end{center}

$2P_4+4P_3$

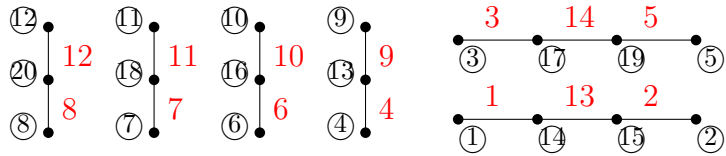
\begin{center}
\begin{picture}(280, 45)
\textcolor{red}{
\put(20,10){8}
\put(20,30){12}
\put(60,10){7}
\put(60,30){11}
\put(100,10){6}
\put(100,30){10}
\put(140,10){4}
\put(140,30){9}
\put(180,15){1}
\put(180,45){3}
\put(210,15){13}
\put(210,45){14}
\put(240,15){2}
\put(240,45){5}
}
\put(0, 5){\footnotesize\textcircled{8}}
\put(0, 25){\footnotesize\textcircled{20}}
\put(0, 45){\footnotesize\textcircled{12}}
\put(40, 5){\footnotesize\textcircled{7}}
\put(40, 25){\footnotesize\textcircled{18}}
\put(40, 45){\footnotesize\textcircled{11}}
\put(80, 5){\footnotesize\textcircled{6}}
\put(80, 25){\footnotesize\textcircled{16}}
\put(80, 45){\footnotesize\textcircled{10}}
\put(120, 5){\footnotesize\textcircled{4}}
\put(120, 25){\footnotesize\textcircled{13}}
\put(120, 45){\footnotesize\textcircled{9}}
\put(170, 0){\footnotesize\textcircled{1}}
\put(200, 0){\footnotesize\textcircled{14}}
\put(230, 0){\footnotesize\textcircled{15}}
\put(260, 0){\footnotesize\textcircled{2}}
\put(170, 30){\footnotesize\textcircled{3}}
\put(200, 30){\footnotesize\textcircled{17}}
\put(230, 30){\footnotesize\textcircled{19}}
\put(260, 30){\footnotesize\textcircled{5}}
\put(15, 5){\circle*{4} }
\put(15, 25){\circle*{4}}
\put(15, 45){\circle*{4} }
\put(15, 5){\line(0, 1){40}}
\put(55, 5){\circle*{4} }
\put(55, 25){\circle*{4}}
\put(55, 45){\circle*{4} }
\put(55, 5){\line(0, 1){40}}
\put(95, 5){\circle*{4} }
\put(95, 25){\circle*{4} }
\put(95, 45){\circle*{4} }
\put(95, 5){\line(0, 1){40}}
\put(135, 5){\circle*{4} }
\put(135, 25){\circle*{4} }
\put(135, 45){\circle*{4} }
\put(135, 5){\line(0, 1){40}}
\put(170, 10){\circle*{4}}
\put(200, 10){\circle*{4}}
\put(230, 10){\circle*{4}}
\put(260, 10){\circle*{4}}
\put(170, 10){\line(1, 0){90}}
\put(170, 40){\circle*{4}}
\put(200, 40){\circle*{4}}
\put(230, 40){\circle*{4}}
\put(260, 40){\circle*{4}}
\put(170, 40){\line(1, 0){90}}
\end{picture}
\end{center}

$2S_3+4P_3$

\begin{center}
\begin{picture}(295, 45)
\textcolor{red}{
\put(20,10){4}
\put(20,30){11}
\put(60,10){3}
\put(60,30){13}
\put(100,10){7}
\put(100,30){10}
\put(140,10){6}
\put(140,30){12}
\put(190,20){1}
\put(180,45){5}
\put(240,15){2}
\put(210,45){14}
\put(270,15){8}
\put(250,25){9}
}
\put(0, 5){\footnotesize\textcircled{4}}
\put(0, 25){\footnotesize\textcircled{15}}
\put(0, 45){\footnotesize\textcircled{11}}
\put(40, 5){\footnotesize\textcircled{3}}
\put(40, 25){\footnotesize\textcircled{16}}
\put(40, 45){\footnotesize\textcircled{13}}
\put(80, 5){\footnotesize\textcircled{7}}
\put(80, 25){\footnotesize\textcircled{17}}
\put(80, 45){\footnotesize\textcircled{10}}
\put(120, 5){\footnotesize\textcircled{6}}
\put(120, 25){\footnotesize\textcircled{18}}
\put(120, 45){\footnotesize\textcircled{12}}
\put(290, 0){\footnotesize\textcircled{8}}
\put(200, 0){\footnotesize\textcircled{1}}
\put(230, 0){\footnotesize\textcircled{2}}
\put(260, 0){\footnotesize\textcircled{19}}
\put(170, 30){\footnotesize\textcircled{3}}
\put(202, 30){\footnotesize\textcircled{20}}
\put(230, 30){\footnotesize\textcircled{14}}
\put(262, 30){\footnotesize\textcircled{9}}
\put(15, 5){\circle*{4} }
\put(15, 25){\circle*{4}}
\put(15, 45){\circle*{4} }
\put(15, 5){\line(0, 1){40}}
\put(55, 5){\circle*{4} }
\put(55, 25){\circle*{4}}
\put(55, 45){\circle*{4} }
\put(55, 5){\line(0, 1){40}}
\put(95, 5){\circle*{4} }
\put(95, 25){\circle*{4} }
\put(95, 45){\circle*{4} }
\put(95, 5){\line(0, 1){40}}
\put(135, 5){\circle*{4} }
\put(135, 25){\circle*{4} }
\put(135, 45){\circle*{4} }
\put(135, 5){\line(0, 1){40}}
\put(290, 10){\circle*{4}}
\put(200, 10){\circle*{4}}
\put(230, 10){\circle*{4}}
\put(260, 10){\circle*{4}}
\put(230, 10){\line(1, 0){60}}
\put(170, 40){\circle*{4}}
\put(200, 40){\circle*{4}}
\put(230, 40){\circle*{4}}
\put(260, 40){\circle*{4}}
\put(170, 40){\line(1, 0){60}}
\put(200, 40){\line(0, -1){30}}
\put(260, 10){\line(0, 1){30}}
\end{picture}
\end{center}

$P_4+S_3+4P_3$

\begin{center}
\begin{picture}(330, 45)
\textcolor{red}{
\put(20,10){6}
\put(20,30){11}
\put(60,10){5}
\put(60,30){13}
\put(100,10){9}
\put(100,30){10}
\put(140,10){8}
\put(140,30){12}
\put(170,15){3}
\put(190,35){7}
\put(190,5){4}
\put(240,30){1}
\put(270,30){14}
\put(300,30){2}
}
\put(0, 5){\footnotesize\textcircled{6}}
\put(0, 25){\footnotesize\textcircled{17}}
\put(0, 45){\footnotesize\textcircled{11}}
\put(40, 5){\footnotesize\textcircled{5}}
\put(40, 25){\footnotesize\textcircled{18}}
\put(40, 45){\footnotesize\textcircled{13}}
\put(80, 5){\footnotesize\textcircled{9}}
\put(80, 25){\footnotesize\textcircled{19}}
\put(80, 45){\footnotesize\textcircled{19}}
\put(120, 5){\footnotesize\textcircled{8}}
\put(120, 25){\footnotesize\textcircled{20}}
\put(120, 45){\footnotesize\textcircled{12}}
\put(195, 22){\footnotesize\textcircled{14}}
\put(215, 0){\footnotesize\textcircled{4}}
\put(290, 15){\footnotesize\textcircled{16}}
\put(320, 15){\footnotesize\textcircled{2}}
\put(160, 30){\footnotesize\textcircled{3}}
\put(215, 45){\footnotesize\textcircled{7}}
\put(230, 15){\footnotesize\textcircled{1}}
\put(260, 15){\footnotesize\textcircled{15}}
\put(15, 5){\circle*{4} }
\put(15, 25){\circle*{4}}
\put(15, 45){\circle*{4} }
\put(15, 5){\line(0, 1){40}}
\put(55, 5){\circle*{4} }
\put(55, 25){\circle*{4}}
\put(55, 45){\circle*{4} }
\put(55, 5){\line(0, 1){40}}
\put(95, 5){\circle*{4} }
\put(95, 25){\circle*{4} }
\put(95, 45){\circle*{4} }
\put(95, 5){\line(0, 1){40}}
\put(135, 5){\circle*{4} }
\put(135, 25){\circle*{4} }
\put(135, 45){\circle*{4} }
\put(135, 5){\line(0, 1){40}}
\put(160, 25){\circle*{4}}
\put(190, 25){\circle*{4}}
\put(260, 25){\circle*{4}}
\put(290, 25){\circle*{4}}
\put(320, 25){\circle*{4}}
\put(160, 25){\line(1, 0){30}}
\put(190, 25){\line(1, 1){20}}
\put(190, 25){\line(1, -1){20}}
\put(210, 5){\circle*{4}}
\put(210, 45){\circle*{4}}
\put(230, 25){\circle*{4}}
\put(260, 25){\circle*{4}}
\put(230, 25){\line(1, 0){90}}
\end{picture}
\end{center}

\caption{Graphs on 20 vertices and 14 edges and the $(1,1)$-antimagic labelings.}
    \label{n20m14}
\end{figure}

\end{document}